\documentclass[graybox]{svmult}

\usepackage{mathptmx}
\usepackage{helvet}
\usepackage{courier}
\usepackage{type1cm}

\usepackage{makeidx}
\usepackage{graphicx}
\usepackage{subfig}
\usepackage{url}

\usepackage{multicol}
\usepackage[bottom]{footmisc}

\usepackage{amsmath,amssymb}

\makeindex

\begin{document}

\title*{Optimal Control of Tuberculosis: A Review\thanks{This is a preprint of a paper 
whose final and definite form will be published in the volume \emph{Mathematics of Planet Earth} 
that initiates the book series \emph{CIM Series in Mathematical Sciences} (CIM-MS) published by Springer.
Submitted Aug 2013; Revised and Accepted June 2014.}}

\author{Cristiana J. Silva and Delfim F. M. Torres}

\authorrunning{C. J. Silva and D. F. M. Torres}

\institute{Cristiana J. Silva
\at
Center for Research and Development in Mathematics and Applications (CIDMA),
Department of Mathematics, University of Aveiro, 3810--193 Aveiro, Portugal;
\email{cjoaosilva@ua.pt}
\and Delfim F. M. Torres
\at
Center for Research and Development in Mathematics and Applications (CIDMA),
Department of Mathematics, University of Aveiro, 3810--193 Aveiro, Portugal;
\email{delfim@ua.pt}}


\maketitle


\abstract{We review the optimal control of systems modeling the dynamics of tuberculosis.
Time dependent control functions are introduced in the mathematical models,
representing strategies for the improvement of the treatment and cure
of active infectious and/or latent individuals. Optimal control theory allows
then to find the optimal way to implement the strategies, minimizing the number
of infectious and/or latent individuals and keeping the cost of implementation
as low as possible. An optimal control problem is proposed and solved,
illustrating the procedure. Simulations show an effective reduction
in the number of infectious individuals.}


\medskip


\medskip



\section{Introduction}

\emph{Mycobacterium tuberculosis} is the cause of most occurrences of tuberculosis (TB)
and is usually acquired via airborne infection from someone who has active TB.
It typically affects the lungs (pulmonary TB) but can affect other sites as well (extrapulmonary TB).
Only approximately 10\% of people infected with M. tuberculosis develop active TB disease.
Therefore, approximately 90\% of people infected remain latent. Latent infected TB people
are asymptomatic and do not transmit TB, but may progress to active TB through either endogenous
reactivation or exogenous reinfection \cite{Small_Fuj_2001,Styblo_1978}.
Following the World Health Organization (WHO), between 1995 and 2011,
51 million people were successfully treated for TB in countries that adopted the WHO strategy,
saving 20 million lives \cite{WHO:2012}. However, the global burden of TB remains enormous.
In 2011, there were an estimated 8.7 million new cases of TB (13\% co-infected with HIV)
and 1.4 million people died from TB \cite{WHO:2012}. The increase of new cases has been attributed
to the spread of HIV, the collapse of public health programs, the emergence of drug-resistant strains
of \emph{M. tuberculosis} \cite{Frieden:et:all:2003,Raviglione:2002,Raviglione:et:all:1997}
and exogenous re-infection, where a latently-infected individual acquires a new infection
from another infectious (see \cite{Bowong_2013,Chiang:2005,Castillo_Chavez_2000}
and references cited therein). In the absence of an effective vaccine, current control programs
for TB have focused on chemotherapy. Lack of compliance with drug treatments not only may lead
to a relapse but to the development of antibiotic resistant TB, called multidrug-resistant TB (MDR-TB),
which is one of the most serious public health problems facing society today \cite{SLenhart_2002}.
The progress in responding to multidrug-resistant TB remains slow. There are critical funding gaps
for TB care and control, which is critical to sustain recent gains, make further progress
and support research and development of new drugs and vaccines \cite{WHO:2012}.

Mathematical models are an important tool in analyzing the spread and control of infectious diseases \cite{Hethcote:SIAM:Review:2000}.
Understanding the transmission characteristics of the infectious diseases in communities,
regions and countries, can lead to better approaches to decrease the transmission of these diseases
\cite{Hethcote:1001models:1994,comSofia:Teresa:CapeVerde,MyID:274}. There are many mathematical dynamic models for TB, see, e.g.,
\cite{Blower_Small_Hopewell_1996,Castillo_Feng_1998,Cohen_Murray_2004,Dye_et_all_1998,Gomes_et_all_2004a,Epidemiology:2012,Vynnycky_Fine_1997}.
Most models consider that there are two different ways to progress to active disease after infection:
``fast progressors'' and ``slow progressors''. It is also considered that only 5 to 10\%
of the infected individuals are fast progressors. The remaining are able to contain the infection
(latent infected individuals) and have a much lower probability to develop active disease by endogenous reactivation.
More recent models also consider the possibility of latent and treated individuals being reinfected,
since it was already recognized that infection and/or disease do not confer full protection \cite{Verver_etall_2005}.
Models show that reinfection can be an important component of TB transmission and can have impact
on the efficacy of interventions \cite{Cohen_Murray_2004,Gomes_et_all_2004a,Rodrigues_et_all_2007,Vynnycky_Fine_1997}.
Here we focus on TB models that consider: development of drug resistant TB \cite{Castillo_Chavez_1997};
exogenous reinfection \cite{Bowong_2010,Bowong_2013,Emvudu_et_all,Castillo_Chavez_2000,Gomes_etall_2007,Okuonghae_2010};
fast and slow progression to infection \cite{Bowong_2010,Bowong_2013,Emvudu_et_all,Gomes_etall_2007};
post-exposure interventions \cite{Gomes_etall_2007}; immigration of infectious individuals \cite{Okuonghae_2010};
and time-dependent parameters \cite{Whang_JTB_2011}. These models can be particularly useful in comparing the effects
of various prevention, therapy and control programs \cite{Hethcote:1001models:1994,Urszula:SIR:DCDS:2011}.
Since a variety of these programs are available, it is a natural objective to design optimal programs
in terms of some pre-assumed criterion. This calls for the application of optimal control tools \cite{livro_Lenhart_2007}.

Optimal control has a long history of being applied to problems in biomedicine, particularly,
to models for cancer chemotherapy \cite{Eisen:1979,Urszula:2compart:JOTA:2002,Urszula:Angiogenic:SIAM:JCO:2007,%
Urszula:OptSuboptumor:JTB:2008,Urszula:SIR:DCDS:2011,Martin:Teo:1994,Swan:MB:1990,Swierniak:IMACS:1988,Swierniak:JBS:1995}.
But until recently, little attention has been given to models in epidemiology
\cite{Behncke:2000,Gaff:Schaefer:2009,Urszula:SIR:DCDS:2011,MR2719552,MyID:263,MyID:283,MyID:265}.
In this paper we review the application of optimal control to TB mathematical models.
The first paper appeared in 2002 \cite{SLenhart_2002},
and considers a mathematical model for TB based on \cite{Castillo_Chavez_1997}
with two classes of infected and latent individuals (infected with typical and resistant strain TB)
where the aim is to reduce the number of infected and latent individuals with resistant TB.
Two control strategies are proposed to achieve the objective: a \emph{case finding}
control measure, referring to the identification of individuals latently infected with typical TB
and who are at high risk of developing the disease and who may benefit from prevention therapy
(reducing the number of latent individuals that develop the disease) \cite{SLenhart_2002,Reichman:2000};
and a \emph{case holding} control, representing the effort that prevents the failure of the treatment
in the typical TB infectious individuals and referring to activities and techniques used to ensure regularity
of drug intake for a duration adequate to achieve a cure (reducing the incidence of acquired drug-resistant TB)
\cite{Chaulet:1983,SLenhart_2002}. In \cite{Haffat_et_all} the authors consider the problem of minimizing
the number of infectious individuals with a control intervention representing the effort
on the prevention of the exogenous reinfection. The authors of \cite{Okuonghae_2010} propose the implementation
of a \emph{case finding} control, representing the fraction of active infectious individuals that are identified
and will be isolated in a facility, for an effective treatment and prevention of contact with susceptible
and latent individuals, and a control measure based on the medical testing/screening
of new immigrants before they are allowed into the population. In \cite{Whang_JTB_2011} three control
interventions are studied with the aim of reducing the number of latent and active infectious individuals:
\emph{distancing} control, representing the effort of reducing susceptible individuals that become infected,
such as, isolation of infectious individuals or educational campaigns; \emph{case finding} control applied
to latent individuals; and \emph{case holding} control for infectious individuals.
In \cite{Bowong_2010,Bowong_2013} \emph{case finding} and \emph{case holding} control measures are proposed
for the minimization of the number of active infected individuals. In \cite{Emvudu_et_all} the authors propose
optimal control strategies for reducing the number of individuals in the class of \emph{the lost to follow up individuals}.
In \cite{Silva:Torres:Naco:2012,Silva:Torres:MBS:2013}, optimal strategies for the minimization of the number
of active TB infectious and persistent latent individuals are proposed.

The study of optimal control strategies produce valuable theoretical results, which can be used to suggest
or design epidemic control programs. Depending on a chosen goal (or goals), various objective criteria
may be adopted \cite{Bowong_2010}. Although the implementation of the control policies, suggested
by the mathematical analysis, can be difficult, they can be a support for the public health authorities
and simulation of optimal control problems applied to mathematical models may become a powerful tool
in their hands (see \cite{Bowong_2010} and references cited therein).

The manuscript is organized as follows.
In Section~\ref{sec:uncont:tb:models}
mathematical models for TB dynamics are reviewed.
They form, after introduction of the control functions, the control system
of the optimal control problems on TB epidemics under consideration.
The models with controls are presented in Section~\ref{sec:cont:tb:models}.
A general optimal control problem is formulated in Section~\ref{sec:opt:cont:prob},
where we explain how to obtain the analytic expression for the optimal controls,
using the Pontryagin minimum principle \cite{Pontryagin_et_all_1962}.
In Section~\ref{sec:numer:methods:simu} we recall the numerical methods used
to compute the optimal controls and associated dynamics.
The main conclusions, derived from the numerical simulations, are resumed.
Finally, in Section~\ref{sec:example}, an example is given, illustrating the effectiveness
of the implementation of the control strategies on a TB control disease.
We end with Section~\ref{sec:conc} of conclusions and future research.


\section{Uncontrolled TB Models}
\label{sec:uncont:tb:models}

Mathematical models have become important tools in analyzing the spread and control
of infectious diseases \cite{Hethcote:1001models:1994}. In this section we present different
mathematical TB models which are, after some modifications, the control system of optimal control
problems on TB epidemics (see Section~\ref{sec:cont:tb:models}).

In an infectious disease model, the total population is divided into epidemiological subclasses.
Some of the standard classes are: susceptible individuals ($S$), latently infected individuals
(infected but not infectious) ($E$), infectious ($I$), and the recovered and cured individuals ($R$).
Eight possible compartmental models, described by their flow patterns,
are: $SI$, $SIS$, $SEI$, $SEIS$, $SIR$, $SIRS$, $SEIR$ and $SEIRS$. For example, in a $SEIRS$ model,
susceptible become exposed in the latent period, then infectious, then recovered with temporary
immunity and then susceptible again when the immunity wears off \cite{Hethcote:1001models:1994}.
Here, we choose to denote the class of latently infected individuals by $L$
and the class of recovered and cured individuals by $T$.

In \cite{Castillo_Chavez_1997} the authors present a $SEIRS$ model for TB.
The latently infected and infectious individuals with typical TB are denoted
by $L_1$ and $I_1$, respectively. The model is given by
\begin{equation}
\label{mod:simple:Castillo:1997}
\begin{cases}
\dot{S}(t) = \Lambda - \beta c S(t) \frac{I_1(t)}{N(t)} - \mu S(t) ,\\[0.2 cm]
\dot{L}_1(t) = \beta c S(t) \frac{I_1(t)}{N(t)} - (\mu + k_1 + r_1)L_1(t)
+ \sigma \beta c T(t)\frac{I_1(t)}{N(t)},\\[0.2 cm]
\dot{I}_1(t) = k_1 L_1(t) - (\mu + r_2 +d_1 )I_1(t), \\[0.2 cm]
\dot{T}(t) = r_1 L_1(t) + r_2 I_1(t) - \sigma \beta c T(t) \frac{I_1(t)}{N(t)} - \mu T(t) ,
\end{cases}
\end{equation}
where $N$ denotes the total population, $N(t) = S(t) + L_1(t) + I_1(t) + T(t)$,
$\Lambda$ is the recruitment rate, $\beta$ and $\sigma \beta$ are the probabilities
that susceptible and treated individuals become infected by one infectious individual
$I_1$ per contact per unit of time, respectively, $c$ is the per-capita contact rate,
$\mu$ is the per-capita natural death rate, $k_1$ is the rate at which an individual
leaves the latent class $L_1$ by becoming infectious, $d_1$ is the per-capita TB
induced death rate, and $r_1$ and $r_2$ are per-capita treatment rates
for latent and infectious individuals, respectively. It is assumed that an individual
can be infected only through contacts with infectious individuals.

In the same paper \cite{Castillo_Chavez_1997}, a two-strain model is presented
which considers resistant TB strain. Two subclasses of the total population are added:
$L_2$ (latent) and $I_2$ (infectious), representing the developmental stages
of resistant strains. It is assumed that $I_2$ individuals can infect $S$,
$L_1$ and $T$ individuals. The model is given by the following system:
\begin{equation}
\label{mod:twostrain:Castillo:1997}
\begin{cases}
\dot{S}(t) = \Lambda - \beta c S(t) \frac{I_1(t)}{N(t)}
- \mu S(t) - \beta^* cS(t)\frac{I_2(t)}{N(t)} ,\\[0.2 cm]
\dot{L}_1(t) = \beta c S(t) \frac{I_1(t)}{N(t)} - (\mu + k_1 + r_1)L_1(t)
+ \sigma \beta c T(t)\frac{I_1(t)}{N(t)} + p r_2 I_1(t)\\
\qquad \quad - \beta^* c L_1(t) \frac{I_2(t)}{N(t)},\\[0.2 cm]
\dot{I}_1(t) = k_1 L_1(t) - (\mu + r_2 +d_1 )I_1(t),\\[0.2 cm]
\dot{L}_2(t) = q r_2 I_1(t) - (\mu + k_2)L_2(t) + \beta^* c
\left( S(t) + L_1(t) + T(t)\right)\frac{I_2(t)}{N(t)},\\[0.2 cm]
\dot{I}_2(t) = k_2 L_2(t) - (\mu + d_2)I_2(t),\\[0.2 cm]
\dot{T}(t) = r_1 L_1(t) + (1 - p -q)r_2 I_1(t) - \sigma \beta c T(t)
\frac{I_1(t)}{N(t)} - \mu T(t) - \beta^* c T(t) \frac{I_2(t)}{N},
\end{cases}
\end{equation}
with $N(t) = S(t) + L_1(t) + I_1(t) + L_2(t) + I_2(t) + T(t)$ and where $\beta^*$
is the probability that treated individuals become infected by one resistant-TB
infectious individual $I_2$ per contact per unit of time, $d_2$ and $k_2$
have similar meanings as $d_1$ and $k_1$ for resistant-TB, and $p + q$
is the proportion of those treated infectious individuals who did not
complete their treatment. The proportion $p$ modifies the rate that departs
from the latent class, and $q r_2 I_1(t)$ gives the rate at which individuals
develop resistant-TB due to an incomplete treatment of active TB.
Therefore, $p \geq 0$, $q \geq 0$ and $p + q \leq 1$.

The results of \cite{Castillo_Chavez_2000} suggest that exogenous reinfection
has a drastic effect on the qualitative dynamics of TB. If we introduce into model
\eqref{mod:simple:Castillo:1997} the term $\rho \beta c L_1 I_1/N$,
which represents exogenous reinfection, we obtain the exogenous reinfection
tuberculosis model developed in \cite{Castillo_Chavez_2000}.
The parameter $\rho$ represents the level of reinfection.
A value of $\rho \in (0, 1)$ implies that reinfection is less likely
than a new infection. In fact, a value of $\rho \in (0, 1)$ implies that
a primary infection provides some degree of cross immunity to exogenous reinfections.
A value of $\rho \in (1, \infty)$ implies that TB infection increases the likelihood
of active TB. The authors take the conservative view that $0 < p < 1$
(see \eqref{mod:Haffat:reinf:2009} in Section~\ref{sec:cont:tb:models}
for the model with controls).

In \cite{Okuonghae_2010} a mathematical model is presented, which takes
into account immigration of infectious individuals as well as isolation of the
infectious individuals for treatment. The model without controls is an extension
of that of \cite{Castillo_Chavez_2000}: one subclass of the total population,
the class of isolated infectious individuals with typical TB, is added.
The corresponding controlled model is given in Section~\ref{sec:cont:tb:models},
by \eqref{mod:Okuonghae:2010}.

In \cite{Bowong_2010,Bowong_2013,Emvudu_et_all} fast and slow progression
to the infectious class are considered and both models consider exogenous reinfection,
chemoprophylaxis of latently infected individuals and treatment of active infected individuals.
In \cite{Bowong_2013} a $SEI$ model is proposed, where the infective class is divided
into two subclasses: diagnosed infectious (those who have an active TB confirmed after an examination in a hospital)
and undiagnosed infectious (i.e., those who have an active TB but not confirmed by an examination in a hospital),
denoted by $I_1$ and $I_2$, respectively. The model in \cite{Bowong_2013} is given by the following system
of ordinary differential equations:
\begin{equation}
\label{mod:Bowong:2013}
\begin{cases}
\dot{S} = \Lambda - \beta \frac{I_1}{N}S - \mu S\, , \\[0.2 cm ]
\dot{L}_1 = (1 - g)\beta \frac{I_1}{N}S + r_2 I_1
+ r_3 I_2 -(1- r_1)\sigma \lambda L_1 -[\mu + k_1(1-r_1)]L_1\, ,\\[0.2 cm]
\dot{I}_1 = g f \beta \frac{I_1}{N}S + h(1-r_1)(k_1+ \sigma \beta
\frac{I_1}{N})L_1 -(\mu + d_1 + r_2)I_1 \, ,\\[0.2 cm]
\dot{I}_2 = g(1-f)\lambda S + (1-h)(1-r_1)(k+\sigma \lambda)E - (\mu + d_3 + r_3)J  \, ,
\end{cases}
\end{equation}
where the fraction $g$ of newly infected individuals are assumed to undergo
a fast progression directly to  TB, while the remainder is latently infected
and enter the latent class $L_1$. Among the newly infected individuals
that undergo a fast progression to TB, a fraction $f$ of them is detected,
and will enter the diagnosed infectious class $I_1$, while the remaining $1-f$
is undetected and will be transferred into the undiagnosed infectious class $I_2$.
In this model $r_2$ is the rate of effective per capita therapy of diagnosed
infectious individuals $I_1$.   It is assumed that undiagnosed infectious
individuals can naturally recover and will be transferred into the latent
class $L_1$ at a constant rate $r_3 < r_2$. Here $\sigma$ is the factor reducing
the risk of infection as a result of acquiring immunity for latently infected
individuals $L_1$. Among latently infected individuals who become infectious,
the fraction $h$ of them is diagnosed and treated, while the remaining $1-h$
is not diagnosed and enters the undiagnosed infectious class $I_2$.
The parameter $d_3$ is the per capita TB induced death rate
for undiagnosed infectious individuals.
If we consider $f=1$, $h=1$, $r_3 = 0$ and $d_3 =0$,
then we obtain the model proposed in \cite{Bowong_2010}.

In \cite{Gomes_etall_2007} the authors present a model for TB that considers
exogenous reinfection and post-exposure interventions. The class $L_3$ denotes
the fraction of early latent individuals, that is, individuals that were recently
infected (less than two years) and are not yet infectious; while $L_4$ denotes
the class of persistent latent individuals who where infected and remain latent.
The other classes are $S$, $I_1$ and $T$, with the same meaning has in the previous models.
The model of \cite{Gomes_etall_2007} is given by the following system:
\begin{equation}
\label{modelGab}
\begin{cases}
\dot{S}(t) = \mu N - \frac{\beta}{N} I_1(t) S(t) - \mu S(t),\\
\dot{L_3}(t) = \frac{\beta}{N} I_1(t)\left( S(t) + \sigma L_4(t)
+ \sigma_R T(t)\right) - (\delta + \tau_1 + \mu)L_3(t),\\
\dot{I}_1(t) = k_1 \delta L_3(t) + \omega L_4(t) + \omega_R T(t)
- (\tau_0  + \mu) I_1(t),\\
\dot{L_4}(t) = (1 - k_1) \delta L_3(t) - \sigma \frac{\beta}{N} I_1(t) L_4(t)
- (\omega + \tau_2 + \mu)L_4(t),\\
\dot{T}(t) = \tau_0 I_1(t) +  \tau_1 L_3(t)
+ \tau_2 L_4(t)
- \sigma_R \frac{\beta}{N} I_1(t) T(t) - \left(\omega_R + \mu\right) T(t) \, .
\end{cases}
\end{equation}
Here $\sigma$ has the same meaning has in the model \eqref{mod:Bowong:2013}
but applies to persistent latent individuals, $L_4$, and $\sigma_R$
represents the same parameter factor but for treated patients;
$\delta$ denotes the rate at which individuals leave the $L_3$ compartment;
$\omega$ is the rate of endogenous reactivation
for persistent latent infections (untreated latent infections);
$\omega_R$ is the rate of endogenous reactivation for treated individuals
(for those who have undergone a therapeutic intervention);
$\tau_0$ is the rate of recovery under treatment of active TB
(assuming an average duration of infectiousness of six months);
$\tau_1$ and $\tau_2$ apply to latent individuals $L_3$ and $L_4$,
respectively, and are the rates at which chemotherapy or a post-exposure
vacine is applied. In this model it is assumed that the total population is constant,
i.e., the rate of birth and death, $\mu$, are equal and there are no disease-related deaths.


\section{Controlled TB Models}
\label{sec:cont:tb:models}

The model \eqref{mod:twostrain:Castillo:1997} is the basis of the work developed
in \cite{SLenhart_2002}, where two control functions, $u_1$ and $u_2$, are introduced,
representing control strategies for the two-strain TB model.
The control system is given by
\begin{equation}
\label{mod:twostrain:SusLenhart:2002}
\begin{cases}
\dot{S}(t) = \Lambda - \beta c S(t) \frac{I_1(t)}{N(t)} - \mu S(t)
- \beta^* cS(t)\frac{I_2(t)}{N(t)} ,\\[0.2 cm]
\dot{L}_1(t) = \beta c S(t) \frac{I_1(t)}{N(t)}
- \left(\mu + k_1 + u_1(t)r_1\right)L_1(t)
+ \sigma \beta c T(t)\frac{I_1(t)}{N(t)}\\
\qquad\quad + (1-u_2(t))p r_2 I_1(t)
- \beta^* c L_1(t) \frac{I_2(t)}{N(t)},\\[0.2 cm]
\dot{I}_1(t) = k_1 L_1(t) - (\mu + r_2 +d_1 )I_1(t),\\[0.2 cm]
\dot{L}_2(t) = (1-u_2(t))q r_2 I_1(t) - (\mu + k_2)L_2(t)
+ \beta^* c \left( S(t) + L_1(t) + T(t)\right)\frac{I_2(t)}{N(t)},\\[0.2 cm]
\dot{I}_2(t) = k_2 L_2(t) - (\mu + d_2)I_2(t),\\[0.2 cm]
\dot{T}(t) = u_1(t)r_1 L_1(t) + \left(1 -((1-u_2(t)))(p +q)\right)r_2 I_1(t)
- \sigma \beta c T(t) \frac{I_1(t)}{N(t)}\\
\qquad\quad - \mu T(t) - \beta^* c T(t) \frac{I_2(t)}{N(t)} \, .
\end{cases}
\end{equation}
The control $u_1$ represents the fraction of typical TB latent individuals, $L_1$,
that is identified and put under treatment (to reduce the number of individuals
that may be infectious). The coefficient $1-u_2(t)$ represents the effort that prevents
the failure of the treatment in the typical TB infectious individuals (to reduce
the number of individuals developing resistant TB). When the control $u_2$ is near 1,
there is low treatment failure and high implementation costs.

In \cite{Haffat_et_all} the authors consider the exogenous reinfection TB model presented
in \cite{Castillo_Chavez_2000} and introduce a control which simulates the effect
of exogenous reinfection, that is, they consider a fixed value for $\rho$, $\rho =0.4$,
and multiply the term $\rho \beta c L_1 I_1/N$ by $1-u$. The coefficient $1 - u$
represents the effort that prevents the exogenous reinfection in order to reduce
the contact between the infectious and exposed individuals, thus decreasing the number
of infectious individuals. The exogenous reinfection TB model
with control, proposed in \cite{Haffat_et_all}, is given by
\begin{equation}
\label{mod:Haffat:reinf:2009}
\begin{cases}
\dot{S}(t) = \Lambda - \beta c S(t) \frac{I_1(t)}{N(t)} - \mu S(t) ,\\[0.2 cm]
\dot{L}_1(t) = \beta c S(t) \frac{I_1(t)}{N(t)} - p \beta c (1-u(t))
L_1(t)\frac{I_1(t)}{N(t)} - (\mu + k_1)L_1(t)
+ \sigma \beta c T(t) \frac{I_1(t)}{N(t)},\\[0.2 cm]
\dot{I}_1(t) = p \beta c (1 -u(t))L_1(t) \frac{I_1(t)}{N(t)}
+ k_1 L_1(t) - (\mu + r_2 +d_1 )I_1(t),\\[0.2 cm]
\dot{T}(t) = r_2 I_1(t) - \sigma \beta c T(t) \frac{I_1(t)}{N(t)} - \mu T(t) ,
\end{cases}
\end{equation}
with $N(t) = S(t) + L_1(t) + I_1(t) + T(t)$.

In \cite{Okuonghae_2010} the model takes into account immigration of infectious individuals
as well as isolation of the infectious for treatment. Two control functions
are considered: $u_1$ and $u_2$. The control $u_1$ accounts for medical
testing/screening of new immigrants, before they are allowed into
the population, while the coefficient $1 - u_1$ is the effort that sustains such
a testing policy. The control $u_2$ is a \emph{case finding} control that represents
the fraction of active individuals that are identified and will be isolated in a
special facility, like a hospital, for effective treatment and prevention of contacts
with susceptible and latent individuals. Hence, the term $1 + u_2$ represents the effort
that sustains the isolation policy. The model with controls is given by
\begin{equation}
\label{mod:Okuonghae:2010}
\begin{cases}
\dot{S} = \Lambda^* + (1 - (1-u_1(t))(p^*+q^*))A - \beta c S \frac{I_1+lJ}{N}-\mu S,\\[0.2 cm]
\dot{L}_1 = (1-u_1(t))p^* A + (1-m)\beta c S \frac{I_1+lJ}{N} -p\beta c L_1 \frac{I_1+lJ}{N}
+ \sigma \beta c T \frac{I_1+\sigma J}{N}\\
\qquad\quad -(k_1 + \mu) L_1, \\[0.2 cm]
\dot{I}_1 = (1-u_1(t))q^* A + m \beta c S \frac{I_1+lJ}{N} + p\beta c L_1
\frac{I_1+lJ}{N} + k_1 L_1 -(\mu+d_3+r_2)I_1 \\
\qquad\quad - (1+u_2(t)) \xi I_1, \\[0.2 cm]
\dot{J} = (1+u_2(t)) \xi I_1 -(r_3+\mu+d_4)J,\\[0.2 cm]
\dot{T} = r_2 I_1 + r_3 J -\sigma \beta c T\frac{I_1+\sigma J}{N} - \mu T .
\end{cases}
\end{equation}
The constant $A$ represents the number of new members arriving into the population,
per unit of time; $p^*$ is the fraction of $A$ arriving infected with latent TB;
and $q^*$ is the fraction of $A$ arriving infected with active TB, so that
$0 \leq p^* + q^* \leq 1$. It is assumed that $1-(p^*-q^*)A$ individuals
are free from the disease. The parameter $\Lambda^*$ is the recruitment rate.
Here the population is replenished from births and immigration; $d_3$ and $d_4$
are the typical TB-induced mortality rates for active TB individuals, that were
not isolated from the population, and for isolated TB cases, respectively;
$r_3$ is the treatment rate for isolated infectious individuals. The parameter
$l$ is the isolation level and lies in the range $0 \leq l \leq 1$, where $l=0$
indicates absolute isolation for active infectious TB cases and $l=1$ indicates
no effective isolation. The parameter $0 \leq \sigma^* \leq 1$ determines
the level of contact that treated individuals have with isolated individuals.
The authors assume that $\sigma^* < l$ and that the treated individuals have
a reduced contact with the isolated infectious group, as some of the treated
individuals are from the $J$ class. By $m$, $0 < m < 1$, it is denoted the fraction
of persons with new infections who develop to TB fast, per unit of time,
while $\xi$ is the rate of isolation. The parameters $\mu$, $\beta$, $c$,
$\sigma \beta$, $k_1$, $p$, $\sigma$ and $r_2$, have the same meaning
as in the previous models (see Table~\ref{table:parameters}).
\begin{table}[!htb]
\centering
\begin{tabular}{|l | l |}
\hline
{\small{Symbol}} & {\small{Description}}  \\
\hline
{\small{$\Lambda$}} & {\small{Recruitment rate}}\\
{\small{$\mu$}} & {\small{Per-capita natural death rate}}\\
{\small{$b$}} & {\small{Effective birth rate}}\\
{\small{$d_1$}} & {\small{Per-capita typical TB induced death rate}}\\
{\small{$d_2$}} & {\small{Per-capita resistant TB induced death rate}}\\
{\small{$\beta$}} & {\small{Rate at which susceptible individuals become infected by an infectious individual}}\\
\ \  & {\small{with typical TB}}\\
{\small{$\beta^*$}} & {\small{Rate at which susceptible individuals become infected by one resistant-TB}}\\
\ \  & {\small{infectious individual}}\\
{\small{$\sigma \beta$}} & {\small{Rate at which treated individuals become infected by an infectious individual}}\\
\ \  & {\small{with typical TB}}\\
{\small{$c$}} & {\small{Per-capita contact rate}}\\
{\small{$k_1$}} & {\small{Rate of progression to active TB}}\\
{\small{$k_2$}} & {\small{Rate of progression to active resistant TB}}\\
{\small{$r$}} & {\small{Per-capita treatment rate}}\\
{\small{$r_1$}} & {\small{Treatment rate of individuals with latent typical TB}}\\
{\small{$r_2$}} & {\small{Treatment rate of individuals with infectious typical TB}}\\
{\small{$r_3$}} & {\small{Treatment rate of undiagnosed infectious individuals}}\\
{\small{$1-s$}} & {\small{Treatment success rate}}\\
{\small{$p$}} & {\small{Level of exogenous reinfection}}\\
{\small{$u + v$}} & {\small{Proportion of treated infectious individuals who did not complete their treatment}}\\
{\small{$g$}} & {\small{Fraction of newly infected individuals that undergo a fast progression}}\\
\ \  & {\small{to the infectious class}}\\
{\small{$f$}} & {\small{Fraction of newly infected individuals that undergo a fast progression to TB}}\\
{\small{$h$}} & {\small{Fraction of infectious individuals that are diagnosed and treated}}\\
{\small{$\sigma$}} & {\small{Factor reducing the risk of infection as a result of acquiring immunity}}\\
\ \  & {\small{for latently infected individuals}}\\
{\small{$\sigma_R$}} & {\small{Factor reducing the risk of infection as a result of acquiring immunity}}\\
\ \  & {\small{for treated individuals}}\\
{\small{$\delta$}} & {\small{Rate at which individuals leave $L_3$ compartment}}\\
{\small{$\alpha$}} & {\small{Non-progress rate from $L_1$ to $I$}}\\
{\small{$\omega$}} & {\small{Rate of endogenous reactivation for persistent latent infections}}\\
{\small{$\omega_R$}} & {\small{Rate of endogenous reactivation for treated infections}}\\
{\small{$\tau_0$}} & {\small{Rate of recovery under treatment of active TB}}\\
{\small{$\tau_1$}} & {\small{Rate of recovery under treatment of latent individuals $L_3$}}\\
{\small{$\tau_2$}} & {\small{Rate of recovery under treatment of latent individuals $L_4$}}\\
{\small{$N$}} & {\small{Total population}}\\
\hline
\end{tabular}
\caption{Parameters that are used in the mathematical models for TB transmission (with and without controls).}
\label{table:parameters}
\end{table}

In \cite{Whang_JTB_2011} the authors modified a model from \cite{Aparicio:JTB:2002}
in order to study the transmission dynamics for TB in South Korea
in the forty years period from 1970 to 2009.
The total population, $N$, is divided into susceptible individuals ($S$), hight-risk latent
($L_1$) that are recently infected but not infectious, active-TB infectious ($I$)
and permanently latent ($L_5$) with low risk. The main difference from the other
TB models is the incorporation of time-dependent parameters. The birth and mortality rates
are assumed as the time-dependent functions $b(t)$ and $\mu(t)$, respectively.
The time-dependent function $k(t)$ is the per-capita rate of progression to active-TB
from the recently latent class $L_1$. Individuals who do not progress from the class
$L_1$ to the class $I$ and those who are treated in the class $L_1$, are moved to the class
$L_5$ at the per-capita rate $\alpha$ and $r(t)$, respectively. The time-dependent
function $s(t)$ is the proportion of treated infectious individuals who did not complete
their treatment; $1 - s(t)$ is the treatment success rate for active tuberculosis.
As previously, the parameter $\beta$ is the number of new infections with active-TB
per unit of time. The authors propose optimal control treatment strategies of TB
in South Korea, for the period from 2010 to 2030, for various possible scenarios.
Since it is not feasible to have the mortality data or the total population data for the future,
the authors used the averaged constant values from the year 2001 to 2009 instead
of using $b(t)$, $\mu(t)$, $s(t)$ and $r(t)$. The estimated time-dependent $k(t)$
from the year 1970 to 2009 is, however, used to find the optimal treatment strategy
for the future. Three time-dependent controls are introduced into the TB system.
The control $u_1(t)$ is the \emph{distancing control} and the coefficient $1-u_1(t)$
represents the effort of reducing susceptible individuals that become infected
by infectious individuals, such as isolation of infectious people or educational
programs/campaigns for healthy control. The \emph{case finding} control, $u_2(t)$,
represents the effort of decreasing the number of individuals that may be infectious,
such as identification through screening of latent individuals who are in high risk
of developing TB and who may benefit from prevention intervention. The \emph{case holding}
control, $1-u_3(t)$, represents the effort of reducing the reinfection individuals,
such as taking care of patients until they complete their treatment.
The control system is given by
\begin{equation*}
\begin{cases}
\dot{S}(t) = b N(t) - \mu S(t) - (1 - u_1(t))
\beta \frac{S(t)}{N(t)}I(t) \, ,\\[0.2 cm]
\dot{L}_1(t) = (1 - u_1(t)) \beta \frac{S(t)}{N(t)}I(t)
- \left(k(t) + u_2(t)\alpha + \mu\right)L_1(t)
+ (1-u_3(t))s r I(t),\\[0.2 cm]
\dot{I}(t) = k(t)L_1(t) - (r + \mu)I(t)\, ,\\[0.2 cm]
\dot{L}_5(t) = (1-(1-u_3(t))s)rI(t) + u_2(t)\alpha L_1(t) - \mu I(t) \, ,
\end{cases}
\end{equation*}
where $N(t) = S(t) + L_1(t) + I(t) + L_5(t)$.

In \cite{Bowong_2010} the author formulates an optimal control problem where one control,
$u$, is introduced in the TB model. The control represents the effort on the chemoprophylaxis
parameter ($r_1$) of latently infected individuals to reduce the number of individuals
that may develop active TB. The model with control is given by \eqref{mod:Bowong:2013}
with $1-u_1 r_1$ instead of $1-r_1$ and considering $f=h=1$ and $d_2=r_3=0$.
In \cite{Bowong_2013}, additionally to the control $u_1$, a second control $u_2$
is included in the model \eqref{mod:Bowong:2013}, which represents the effort
on detection ($h$) of infectious, to increase the treatment rate of infectious and,
consequently, to reduce the number of infectious and the source of infection.
The model with controls is given by \eqref{mod:Bowong:2013} with $1-u_1 r_1$
instead of $1-r_1$ and $u_2 h$ instead of $h$.

In \cite{Emvudu_et_all} the authors propose a model adapted to Africa, in particular
to Cameroon. A new class of individuals, called \emph{the lost to follow up individuals},
is introduced. The individuals in this class are active infectious individuals who didn't
take the treatment until the end, due to a brief relief of a long time treatment.
Some of the lost to follow up individuals can transmit the disease without presenting
any symptom. The authors present control measures for the reduction of the number
of individuals that progress to the class of the lost to follow up individuals, $L$.

In \cite{Silva:Torres:Naco:2012,Silva:Torres:MBS:2013} two control functions, $u_1$ and $u_2$,
and two real positive constants, $\epsilon_1$ and $\epsilon_2$, were introduced
in the model \eqref{modelGab}. The control $u_1$ represents the effort in preventing
the failure of treatment in active TB infectious individuals $I_1$ (\emph{case holding}),
and the control $u_2$ governs the fraction of persistent latent individuals $L_4$ that is
put under treatment (\emph{case finding}). The parameters $\epsilon_i \in (0, 1)$, $i=1, 2$,
measure the effectiveness of the controls $u_i$, $i=1, 2$, respectively, \textrm{i.e.},
these parameters measure the efficacy of treatment interventions for active and persistent
latent TB individuals, respectively. In \cite{Silva:Torres:Naco:2012} the model is applied to Angola.

In \cite{SLenhart_2002,Silva:Torres:Naco:2012,Silva:Torres:MBS:2013} it is assumed that
the total population $N$ is constant, that is, the recruitment rate is equal to $\mu N$,
$\Lambda = \mu N$, and the TB induced death rates are equal to zero. In
\cite{Bowong_2010,Bowong_2013,Emvudu_et_all,Haffat_et_all,Okuonghae_2010,Whang_JTB_2011}
the total population is not considered to be constant.


\section{Optimal Control Problems}
\label{sec:opt:cont:prob}

The control strategies for the reduction of infectious and/or latent individuals
imply a cost of implementation. This implementation cost depends on many factors,
for example, costs for activities to facilitate \emph{case holding}. Those activities
can be challenging because of the fact that chemotherapy must be maintained for several
months to ensure a lasting cure, but patients usually recover their sense of
well-being after only a few weeks of treatment and may often stop taking medications
\cite{SLenhart_2002,Reichman:2000}. For \emph{case finding}, the control policies
consider actions for the prevention of disease development with preventive therapy
of latently infected individuals, which can be done in different ways, for example,
identifying TB cases where the first initiative patient/provider contact is taken
by health providers (\emph{active case finding}) or by the patient
(\emph{passive case finding}), and screening activities among population groups
at high risk of TB (for example, immigrants from high prevalence countries)
\cite{SLenhart_2002,Okuonghae_2010}. The implementation cost is taken
into account in the formulation of an optimal control problem and is mathematically
traduced by a functional.

Let $L$ and $I$ denote the latent infected and infectious individuals, respectively,
without any specific characteristic, and $u = (u_1, \ldots, u_n)$, with $n \in \{ 1, 2, 3\}$
for the models described in Section~\ref{sec:cont:tb:models}, be the bounded Lebesgue
measurable control function. Different cost functionals have been considered
on the previously cited works on optimal control applied to TB models:
\begin{equation*}
C_1(u) = \int_0^{t_f} \left[A_1 I(t) + A_2 L(t) + \sum_{i=1}^n\frac{B_i}{2}u_i^2(t) \right]  dt \, ,
\end{equation*}
\begin{equation*}
C_2(u) = \int_0^{t_f} \left[ A_1 I(t) + \sum_{i=1}^n\frac{B_i}{2}u_i^2(t) \right]  dt \, ,
\end{equation*}
and
\begin{equation*}
C_3(u) = \int_0^{t_f} \left[ A_2 L(t) + \sum_{i=1}^n\frac{B_i}{2}u_i^2(t) \right]  dt \, .
\end{equation*}
It is assumed that the cost of the treatments are nonlinear and take a quadratic form.
The coefficients, $A_j$, $j \in \{1, 2 \}$, and $B_i$, $i \in \{1, 2, 3 \}$,
are balancing cost factors due to the size and importance
of the three parts of the objective functional.

For the cost functional $C_2$ and $C_3$, the aim is to minimize the infectious
and latent individuals, respectively, while keeping the cost low. For the cost
functional $C_1$, both infectious and latent individuals are wished to be minimized,
keeping the cost of control interventions low.

A cost functional of type $C_1$ is adopted by
\cite{SLenhart_2002,Silva:Torres:Naco:2012,Silva:Torres:MBS:2013,Whang_JTB_2011},
$C_2$ is chosen in \cite{Bowong_2010,Bowong_2013,Haffat_et_all,Okuonghae_2010}
and $C_3$ is the objective functional in \cite{Emvudu_et_all}.

Let $(\mathcal{S})$ denote a mathematical model for TB with controls
(see Section~\ref{sec:cont:tb:models}) given by a finite number, $m$, of differential equations.
Assume that the control system $(\mathcal{S})$ is given by $\dot{X}= f(X, u)$,
where $f$ is a Lipschitz continuous function with respect to the state variable $X$,
$X \in \mathbb{R}^m$, on the time interval $[0, t_f]$ and $X(0) = X_0$
be the initial condition. Moreover, let $g(X, u)$ denote the integrand
of the cost functional $C$ under consideration
and assume that $g$ is convex with respect to the control $u$.
The optimal control problem consists in finding a control $u^*$
such that the associated state trajectory $X^*$ is solution of the control
system $(\mathcal{S})$, in the time interval $[0, t_f]$ with initial
conditions $X^*(0)$, and minimizes the cost functional $C$,
\begin{equation}
\label{eq:cost:min}
C(u^*) = \min_{\Omega} C(u) \, ,
\end{equation}
where $\Omega$ is the set of admissible controls
(bounded and Lebesgue integrable functions) given by
\begin{equation*}
\Omega = \left\{u \in L^1(0, t_f) \, | \, 0 \leq u_i \leq 1 \, , \, \, i=1, \ldots, n  \right\}.
\end{equation*}

According to the Pontryagin minimum principle \cite{Pontryagin_et_all_1962},
if $u^*(\cdot) \in \Omega$ is optimal for the optimization problem
\eqref{eq:cost:min} subject to the control system $(\mathcal{S})$
with fixed initial conditions $X_0$ and fixed final time $t_f$,
then there exists a nontrivial absolutely continuous mapping
$\lambda : [0, t_f] \to \mathbb{R}^m$,
called the \emph{adjoint vector}, such that
\begin{equation}
\label{control:system:OCP}
\dot{X} = \frac{\partial H}{\partial \lambda}
\end{equation}
and
\begin{equation}
\label{adjointsystem:OCP}
\dot{\lambda}= -\frac{\partial H}{\partial X} \, ,
\end{equation}
where the function $H$ defined by
\begin{equation*}
H= H(X(t), \lambda(t), u(t)) =  g\left(X(t), u(t)\right)
+ \langle \lambda(t), f(X(t), u(t) \rangle
\end{equation*}
is called the \emph{Hamiltonian}, and the minimization condition
\begin{equation}
\label{maxcondPMP}
H(X^*(t),\lambda^*(t), u^*(t)) = \min_{0 \leq u \leq 1}
H(X^*(t), \lambda^*(t), u)
\end{equation}
holds almost everywhere on $[0, t_f]$.
Moreover, the transversality conditions
\begin{equation}
\label{transversality:cond:OCP}
\lambda_i(t_f) = 0, \quad
i =1,\ldots, m \, ,
\end{equation}
hold. This approach was considered in
\cite{Bowong_2013,Emvudu_et_all,Haffat_et_all,SLenhart_2002,Okuonghae_2010,Silva:Torres:Naco:2012,Whang_JTB_2011}
for obtaining an analytic expression of the optimal control $u^*$. In \cite{Bowong_2010} the analytical expression
of the optimal control $u^*$ is derived, using an algebraic approach, by solving a Riccati equation.


\section{Numerical Methods and Simulations}
\label{sec:numer:methods:simu}

In \cite{SLenhart_2002} the optimal treatment strategy is obtained by solving the optimality system,
consisting of 12 ODEs from \eqref{mod:twostrain:SusLenhart:2002} and adjoint equations
\eqref{adjointsystem:OCP}. An iterative method is used for solving the optimality system.
The authors start to solve the state equations with a guess for the controls over the simulated
time using a forward fourth order Runge--Kutta scheme. Because of the transversality conditions
\eqref{transversality:cond:OCP}, the adjoint equation are solved by a backward fourth order
Runge--Kutta scheme using the current iteration solution of the state equations. Then, the controls
are updated by using a convex combination of the previous controls and the value from the
characterizations derived by \eqref{maxcondPMP}. This process is repeated and iteration
is stopped if the values of unknowns at the previous iteration are close enough
to the ones at the present iteration. The same numerical procedure is applied in
\cite{Emvudu_et_all,Okuonghae_2010,Whang_JTB_2011}.
In \cite{Silva:Torres:Naco:2012,Silva:Torres:MBS:2013} the authors use also
the software IPOPT \cite{IPOPT}, the Matlab Optimal Control Software
PROPT \cite{PROPT} and the algebraic modeling language AMPL \cite{AMPL}.
See, for example, \cite{Anita:book:OC:Matlab:2011} for details on numerical simulations
of optimal control applied to life sciences using Matlab. In \cite{Haffat_et_all} the authors
apply a semi-implicit finite diference method developed by \cite{Gumel_el_all_2001}
and presented in \cite{Karrakchou_el_all_2006}.
For a gentle overview see \cite{MyID:287}.

In \cite{SLenhart_2002} different optimal control strategies are presented,
which depend on the population size, cost of implementing treatment controls
and the control parameters. The authors conclude that programs that follow
the proposed control strategies can effectively reduce the number of latent
and infectious resistant-strain TB cases. In \cite{Haffat_et_all} the numerical
results show the effectiveness to introduce the control that prevents the exogenous
reinfection, which reactivates the bacterium tuberculosis at the latent individuals.
Analogously, in \cite{Bowong_2010,Bowong_2013} the results emphasize the importance
of controlling exogenous reinfection using chemoprophylaxis and detection methods
in reducing the number of actively infected individuals with tuberculosis.
The numerical simulations in \cite{Okuonghae_2010} show that the proposed control
interventions can effectively reduce the number of latent and infectious TB cases.
More precisely, the optimal control results show that a cost effective combination
of screening/medical testing of immigrants, as well as isolation of infectious persons
for treatment, may depend on cost of implementation of the controls and the parameters
of the model, specially, the rate of isolation $\xi$, isolation level $l$, fraction
of immigrants with latent TB $p$, and fraction of immigrants with active TB $q$.


\section{Example: Optimal Control for the TB SEIRS Model}
\label{sec:example}

In this section we introduce a \emph{case finding} control function $u$ to the SEIRS mathematical
model for TB \eqref{mod:simple:Castillo:1997} from \cite{Castillo_Chavez_1997}.
The coefficient $1-u(t)$ represents the effort that sustains the success of the treatment
of latent individuals $L_1$. We assume that the total population $N$ is constant, that is,
$d_1 = 0$. This assumption is appropriate when the time period is short or when the natural
deaths or the immigration balances the emigration (see \cite{Hethcote:1001models:1994}).

The controlled model is given by (see Table~\ref{table:parameters} for the meaning of the parameters)
\begin{equation}
\label{mod:simple:Castillo:1997:control}
\begin{cases}
\dot{S}(t) = \Lambda - \frac{\beta}{N} c S(t) I_1(t) - \mu S(t) ,\\[0.2 cm]
\dot{L}_1(t) = \frac{\beta}{N} c S(t) I_1(t) - (\mu + r_1) L_1(t) - (1-u(t))k_1 L_1(t)
+ \sigma \frac{\beta}{N} c T(t) I_1(t),\\[0.2 cm]
\dot{I}_1(t) = (1-u(t))k_1 L_1(t) - (\mu + r_2 +d_1 )I_1(t), \\[0.2 cm]
\dot{T}(t) = r_1 L_1(t) + r_2 I_1(t) - \sigma \frac{\beta}{N} c T(t) I_1(t) - \mu T(t) .
\end{cases}
\end{equation}
Our aim is to minimize the number of infectious individuals $I_1$, while keeping the cost
of control strategies implementation low, that is, (we choose a cost functional
of type $C_2$ of Section~\ref{sec:opt:cont:prob})
\begin{equation}
\label{cost:example}
C(u) = \int_0^{t_f} \left[ A I_1(t) + \frac{B}{2}u^2(t) \right]  dt \, .
\end{equation}
In this example we propose to solve the optimal control problem that consists
in finding a control $u^*$ such that the associated state trajectory
$(S^*, L_1^*, I_1^*, T^*)$ is solution of the control system
\eqref{mod:simple:Castillo:1997:control} in the time interval $[0, t_f]$
with initial conditions $(S(0), L_1(0), I_1(0), T(0))$
and minimize the cost functional $C$,
\begin{equation}
\label{eq:cost:min:example}
C(u^*) = \min_{\Omega} C(u) \, ,
\end{equation}
where $\Omega$ is the set of admissible controls given by
\begin{equation*}
\Omega = \{u \in L^1(0, t_f) \, | \, 0 \leq u \leq 1 \, \} \, .
\end{equation*}

\begin{theorem}
\label{the:thm}
The optimal control problem \eqref{mod:simple:Castillo:1997:control}, \eqref{eq:cost:min:example}
with fixed initial conditions $S(0)$, $I_1(0)$, $L_1(0)$ and $T(0)$ and fixed
final time $t_f$, admits an unique solution
$\left(S^*(\cdot), I_1^*(\cdot), L_1^*(\cdot), T^*(\cdot)\right)$
associated to an optimal control $u^*(\cdot)$ on $[0, t_f]$.
Moreover, there exists adjoint functions $\lambda_1^*(\cdot)$, $\lambda_2^*(\cdot)$,
$\lambda_3^*(\cdot)$ and $\lambda_4^*(\cdot)$ such that
\begin{equation}
\label{adjoint_function:example}
\begin{cases}
\dot{\lambda^*_1}(t) =  \lambda^*_1(t) \left( \frac{\beta}{N} c I_1^*(t)+\mu \right)
- \lambda_2^*(t) \frac{\beta}{N} c I_1^*(t), \\[0.1 cm]
\dot{\lambda^*_2}(t) = \lambda_2^*(t)((\mu + r_1) + (1-u^*(t))k_1)
- \lambda^*_3(t) (1-u^*(t))k_1 - \lambda_4^*(t) r_1,\\[0.1 cm]
\dot{\lambda^*_3}(t) = -A +\lambda_1^*(t) \frac{\beta}{N} c S^*(t)
-\lambda_2^*(t)\left(\frac{\beta}{N} c S^*(t) + \sigma \frac{\beta}{N} c T^*(t) \right)\\
\qquad\quad + \lambda_3^*(t)(\mu + r_2 + d_1)
- \lambda_4^*(t)\left(r_2 - \sigma \frac{\beta}{N} c T*(t)) \right), \\[0.1 cm]
\dot{\lambda^*_4}(t) =-\lambda_2^*(t) \sigma \frac{\beta}{N} c I_1^*(t)
+ \lambda_4^*(t) \left( \sigma \frac{\beta}{N} c I_1^*(t) - \mu \right) \, ,
\end{cases}
\end{equation}
with transversality conditions
\begin{equation*}
\lambda^*_i(t_f) = 0,
\quad i=1, \ldots, 4 \, .
\end{equation*}
Furthermore,
\begin{equation}
\label{optcontrols:example}
u^*(t) = \min \left\{ \max \left\{0, \frac{k_1}{B}L_1^*(t) \left(\lambda^*_3(t)
- \lambda^*_2(t)\right)\right\}, 1 \right\} \, .
\end{equation}
\end{theorem}

\begin{proof}
Existence of an optimal solution $\left(S^*, L_1^*, I_1^*, T^*\right)$,
associated to an optimal control $u^*$, comes from
the convexity of the integrand of the cost functional \eqref{cost:example} with respect
to the control $u$ and the Lipschitz property of the state system
with respect to state variables $\left(S, L_1, I_1, T\right)$
(see, \textrm{e.g.}, \cite{Cesari_1983,Fleming_Rishel_1975}).
System \eqref{adjoint_function:example} is derived from the Pontryagin minimum principle
(see \eqref{adjointsystem:OCP}, \cite{Pontryagin_et_all_1962})
and the optimal control \eqref{optcontrols:example}
comes from the minimization condition \eqref{maxcondPMP}.
The optimal control given by \eqref{optcontrols:example}
is unique along all time interval due to the boundedness
of the state and adjoint functions, the Lipschitz property
of systems \eqref{mod:simple:Castillo:1997:control}
and \eqref{adjoint_function:example} and the fact that the problem is autonomous.
\end{proof}

We end by presenting some numerical simulations with the following parameter values:
$\mu = 0.0143$, $c = 1$, $\beta = 13$, $\sigma = 1$, $r_1 = 2$, and $r_2 = 1$
(see \cite{Castillo_Chavez_1997}). The initial conditions are: $S(0)=(76/120)N$,
$L_1(0) = (38/120)N$, $I_1(0) = (5/120)N$, and $T(0) = (1/120)N$ (see \cite{SLenhart_2002}).
We start showing that the implementation of the control has a positive impact
on the reduction of infectious individuals. In Figure~\ref{fig:I1:T:with:without:cont}
we observe that the fraction of infectious individuals decreases significatively
when control strategies are implemented.
\begin{figure}[!htb]
\centering
\subfloat[\footnotesize{$I_1/N$}]{\label{I1:with:without:u}
\includegraphics[width=0.51\textwidth]{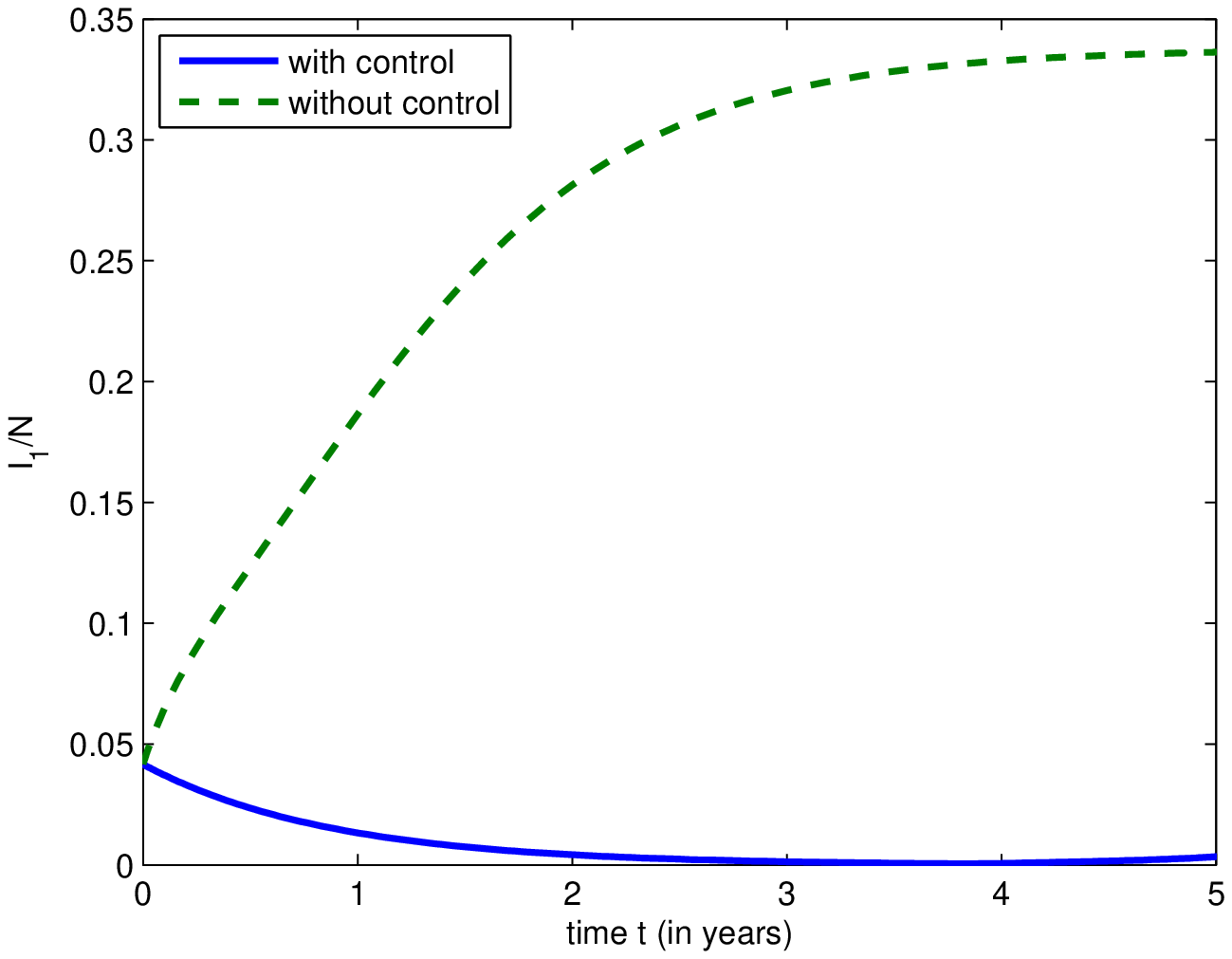}}
\subfloat[\footnotesize{Optimal control $u$}]{\label{u:with:without:u}
\includegraphics[width=0.51\textwidth]{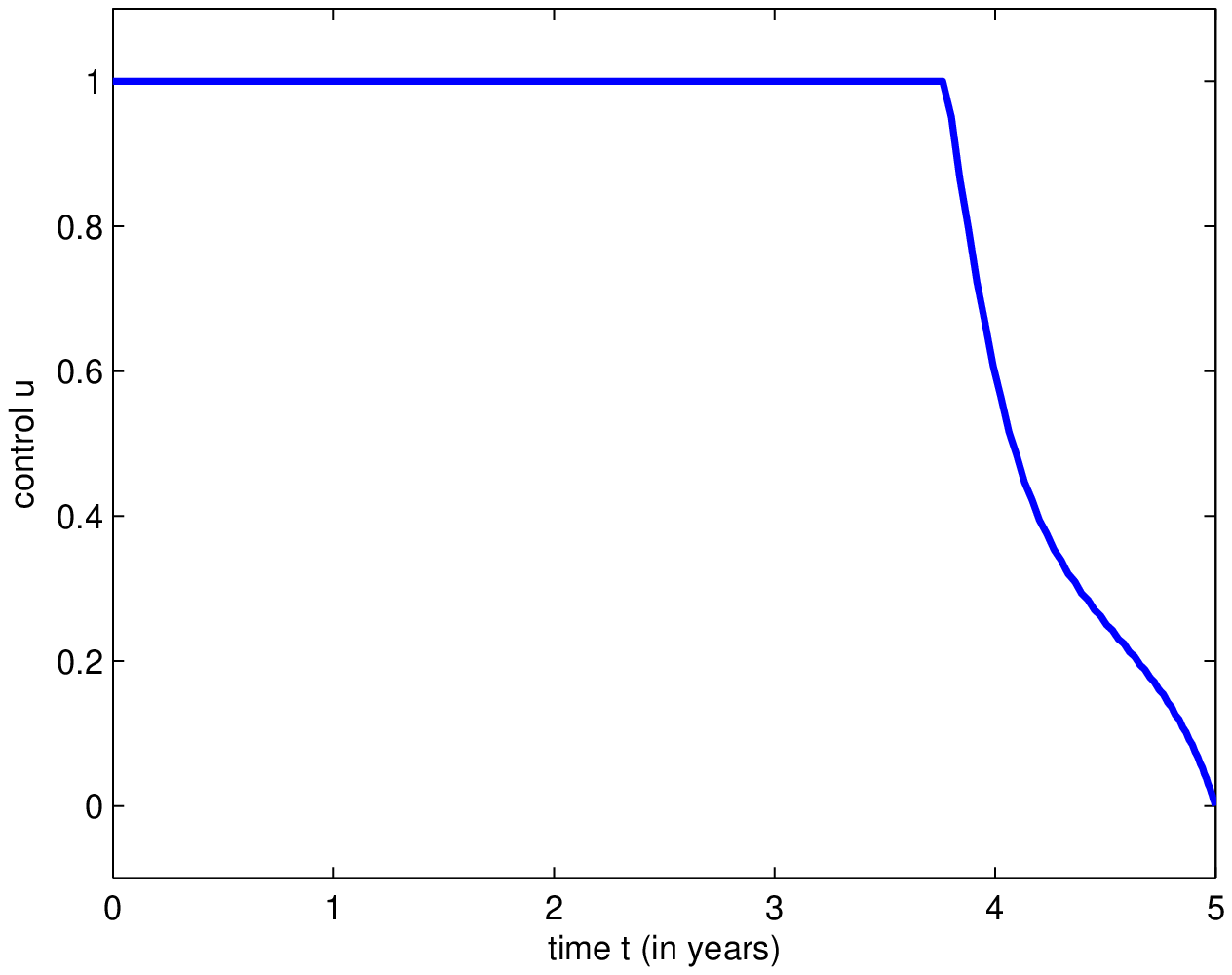}}
\caption{Fraction of infectious individuals, with and without control,
and optimal control (for $k_1 = 1$, $A = 1$, $B = 100$ and $N = 10000$).}
\label{fig:I1:T:with:without:cont}
\end{figure}
If our aim is to reduce the number of infectious individuals giving special
attention to keep the cost of implementation of the control measures low,
then the weight constant $B$ should take bigger values than $A$. Take, without loss of generality,
$A = 1$ and $B \geq 50$. In this case, we observe that the fraction of infectious individuals $I_1/N$
and the optimal control $u$ depend on the rate of progression to active TB
(see Figure~\ref{fig:I1:u:varK}) and the size $N$ of total population (see Figure~\ref{fig:I1:u:varN}).
The period of time that the optimal control attains its maximum value increases with $B$
(see Figure~\ref{fig:I1:u:varB}). However, contrary to what is desired, the fraction
of infectious individuals starts increasing after some specific period of time.
This can be avoided if the rate $k$ of progression to active TB is low (see Figure~\ref{fig:I1:u:varK}),
or if we give more importance to the decrease of the number of infectious individuals
than to the cost of implementation of the control policies, that is, if we increase the value
of the weight constant $A$. In fact, for $A \geq B$ the fraction of infectious individuals
never increases in all treatment period, regardless the size of the population $N$
or the value of $k$ (see Figure~\ref{fig:I:A100:varkN}). On the other hand,
the optimal control attains the maximal value almost all the treatment period,
which implies a higher cost implementation of control measures (see Figure~\ref{fig:u:A100:varkN}).


\begin{figure}[!htb]
\centering
\subfloat[\footnotesize{$I_1/N$}]{\label{I1:varK}
\includegraphics[width=0.51\textwidth]{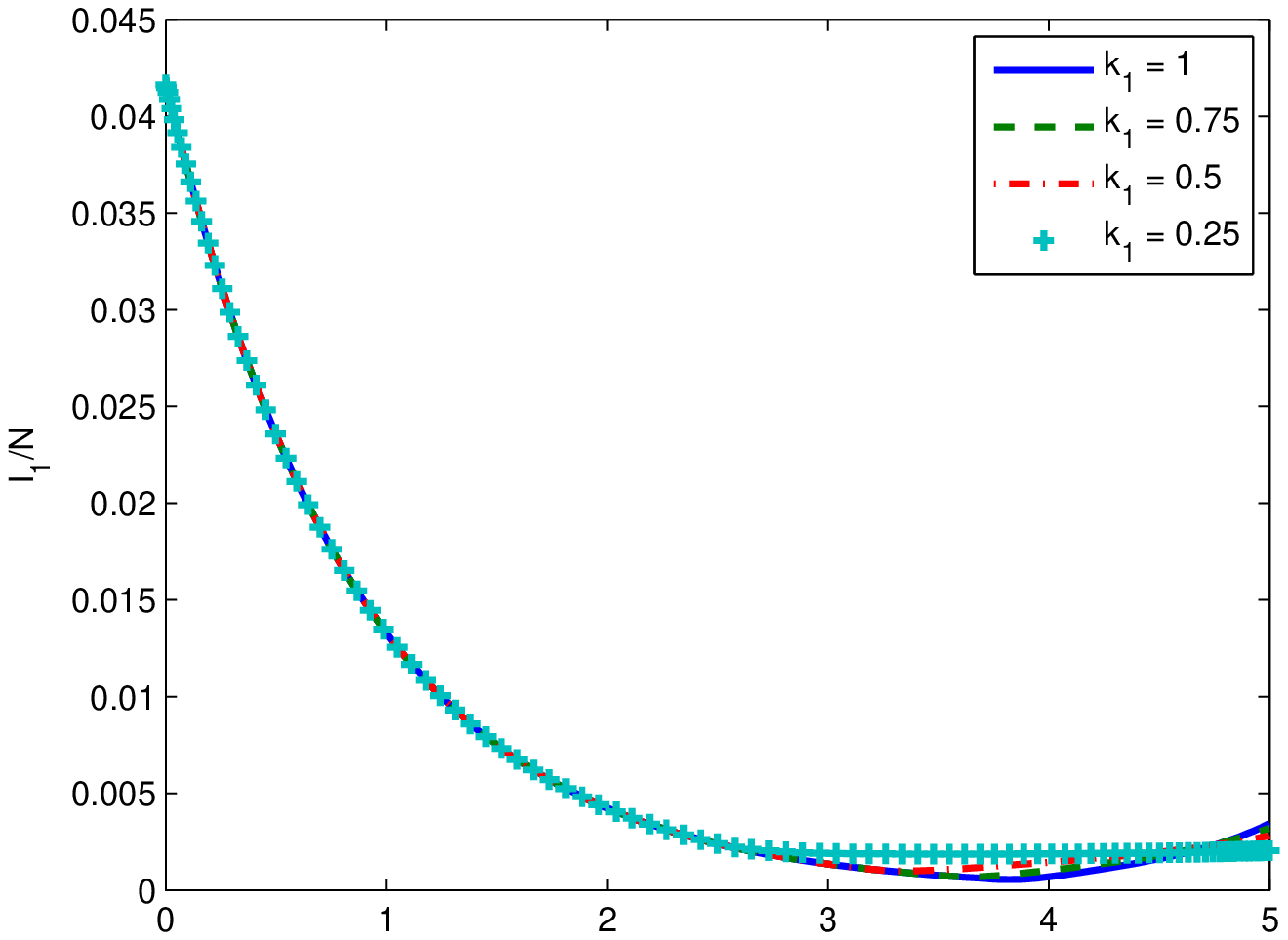}}
\subfloat[\footnotesize{Optimal control $u$}]{\label{u:varK}
\includegraphics[width=0.51\textwidth]{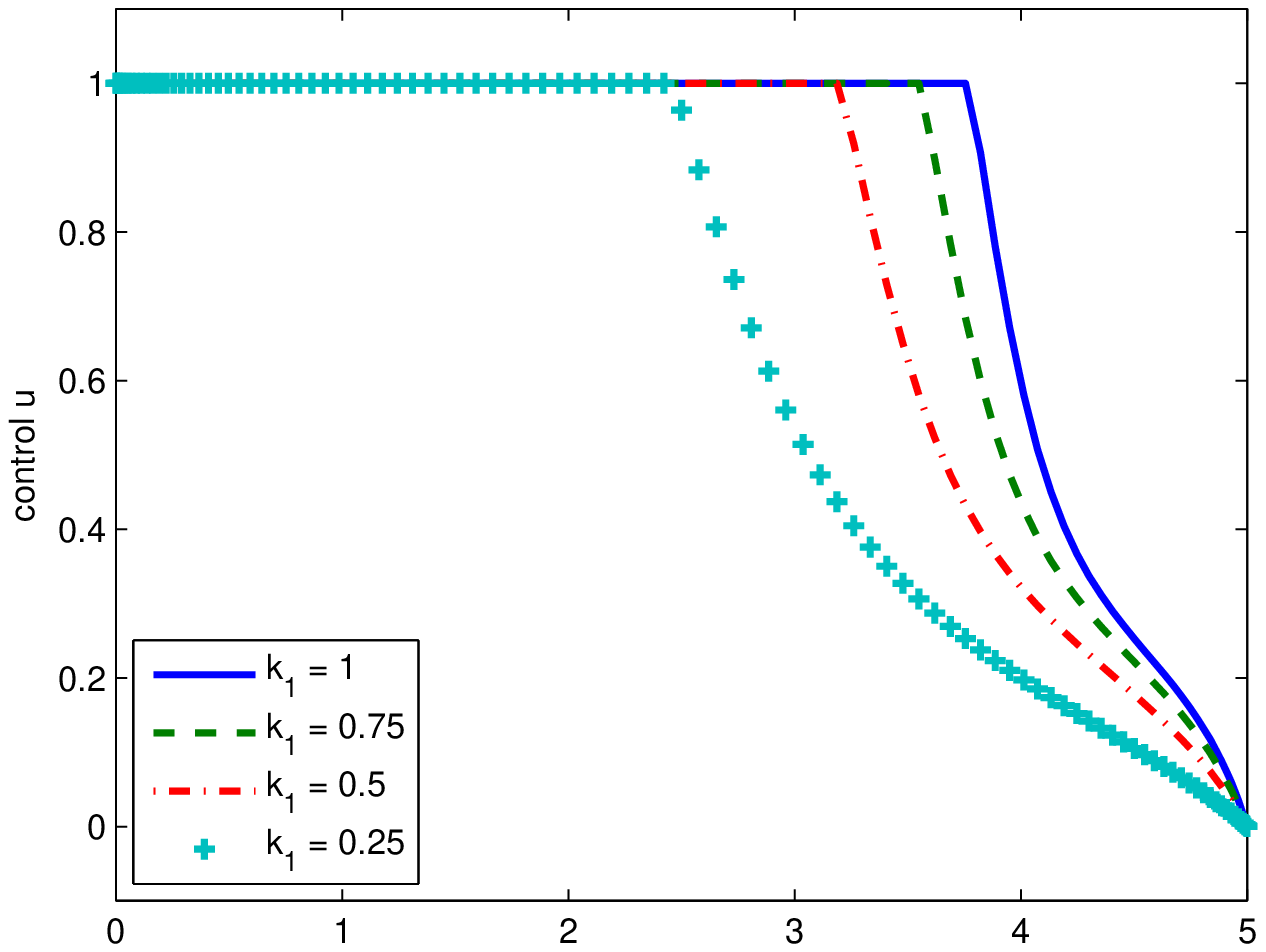}}
\caption{Fraction of infectious individuals and optimal control for
$k_1 \in \{0.25, 0.5, 0.75, 1\}$ (with $B = 100$, $A = 1$ and $N = 10000$).}
\label{fig:I1:u:varK}
\end{figure}


\begin{figure}[!htb]
\centering
\subfloat[\footnotesize{$I_1/N$}]{\label{I1:varN}
\includegraphics[width=0.51\textwidth]{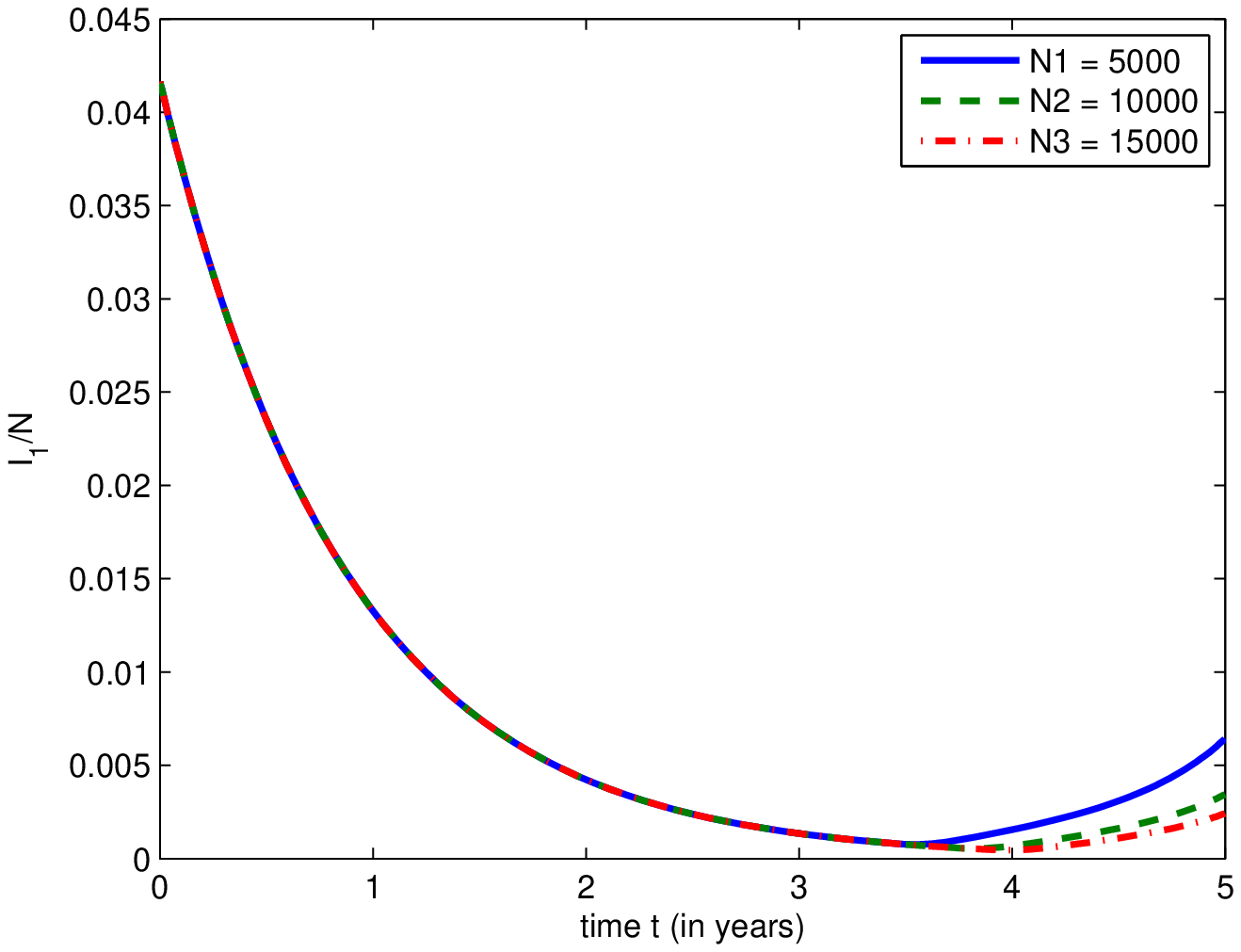}}
\subfloat[\footnotesize{$u$}]{\label{u:varN}
\includegraphics[width=0.51\textwidth]{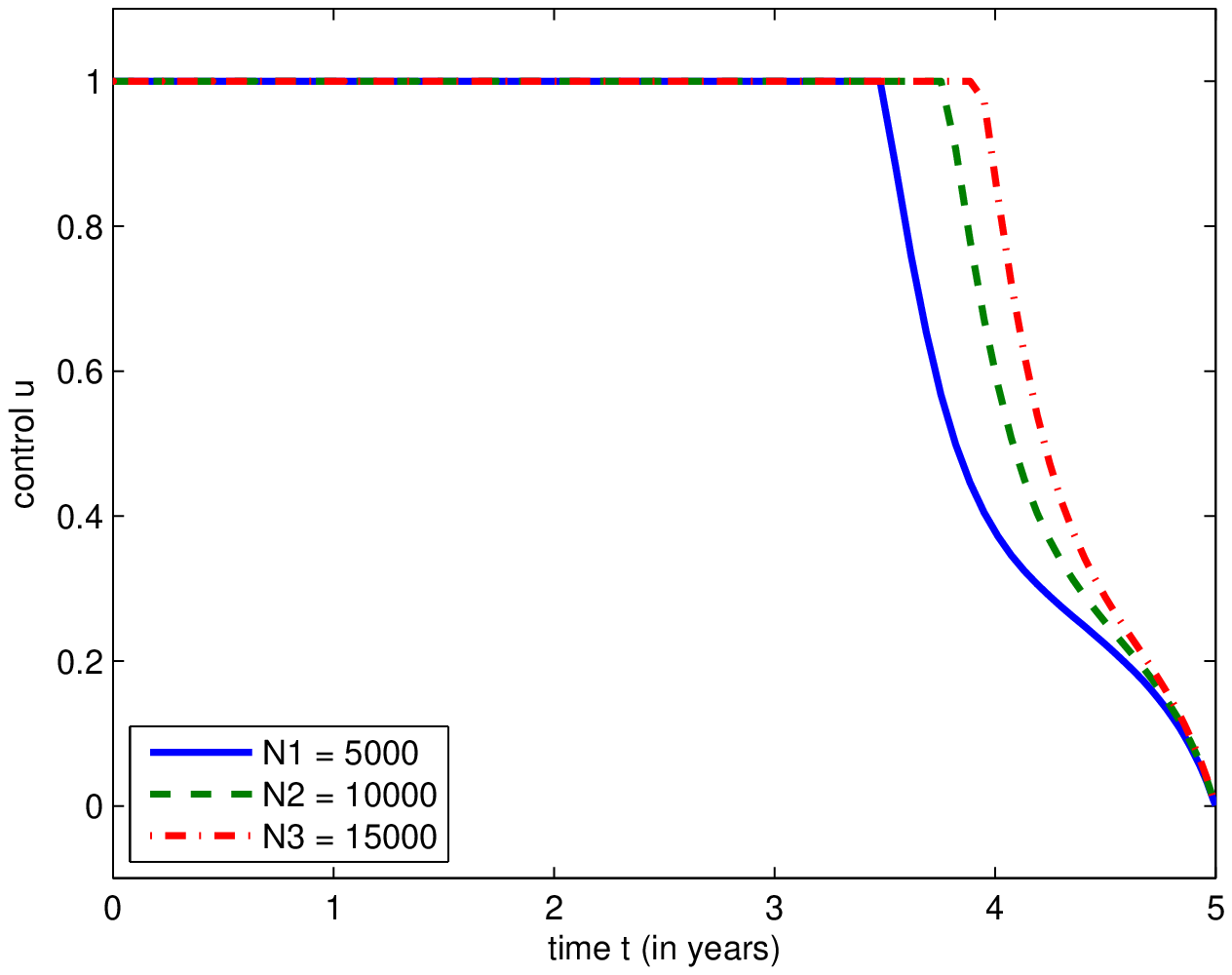}}
\caption{Fraction of infectious individuals and optimal control for
$N \in \{5000, 10000, 15000\}$ (with $B = 100$, $A = 1$ and $k_1 = 1$).}
\label{fig:I1:u:varN}
\end{figure}


\begin{figure}[!htb]
\centering
\subfloat[\footnotesize{}]{\label{I1:varB}
\includegraphics[width=0.51\textwidth]{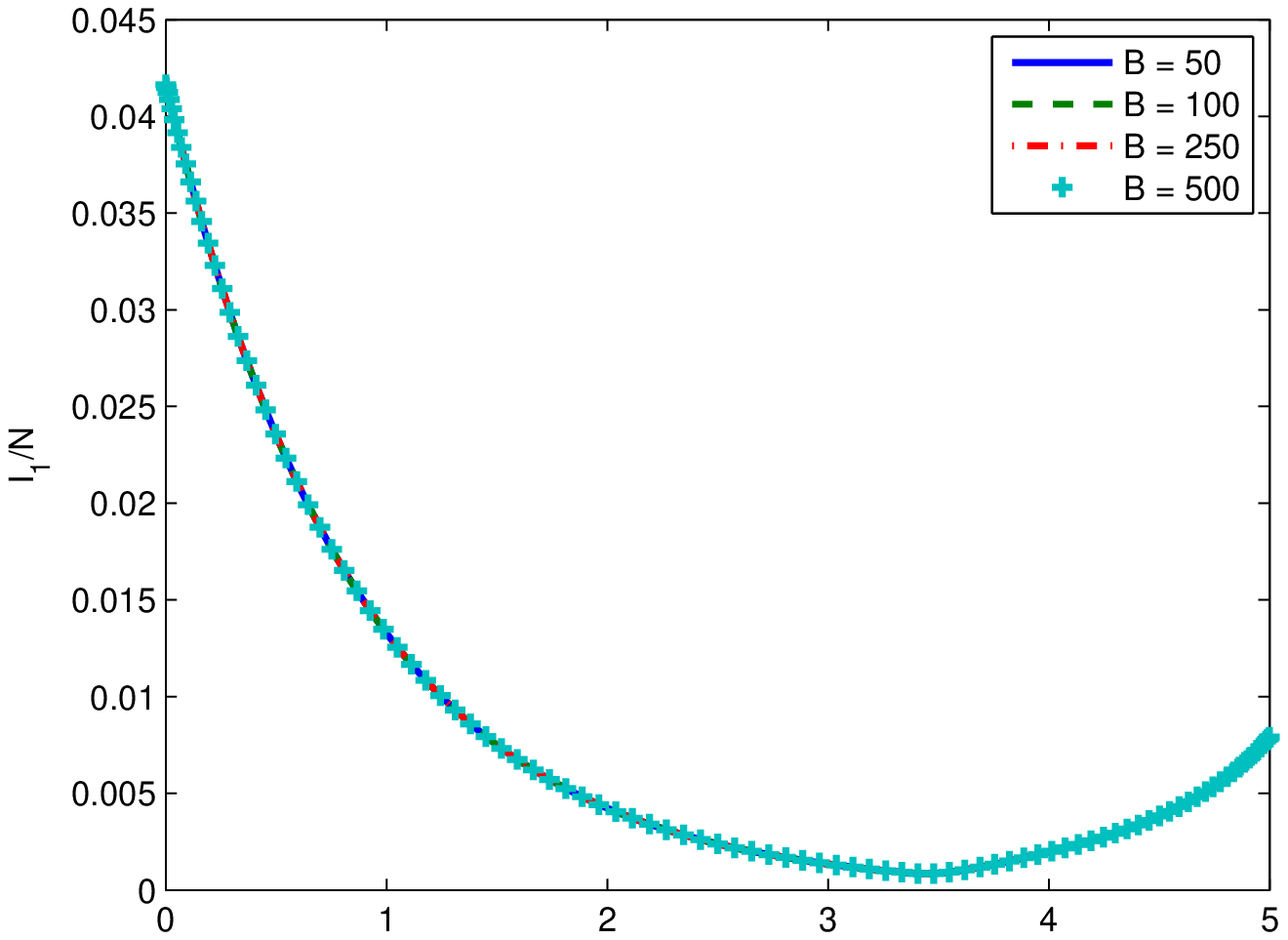}}
\subfloat[\footnotesize{}]{\label{u:varB}
\includegraphics[width=0.51\textwidth]{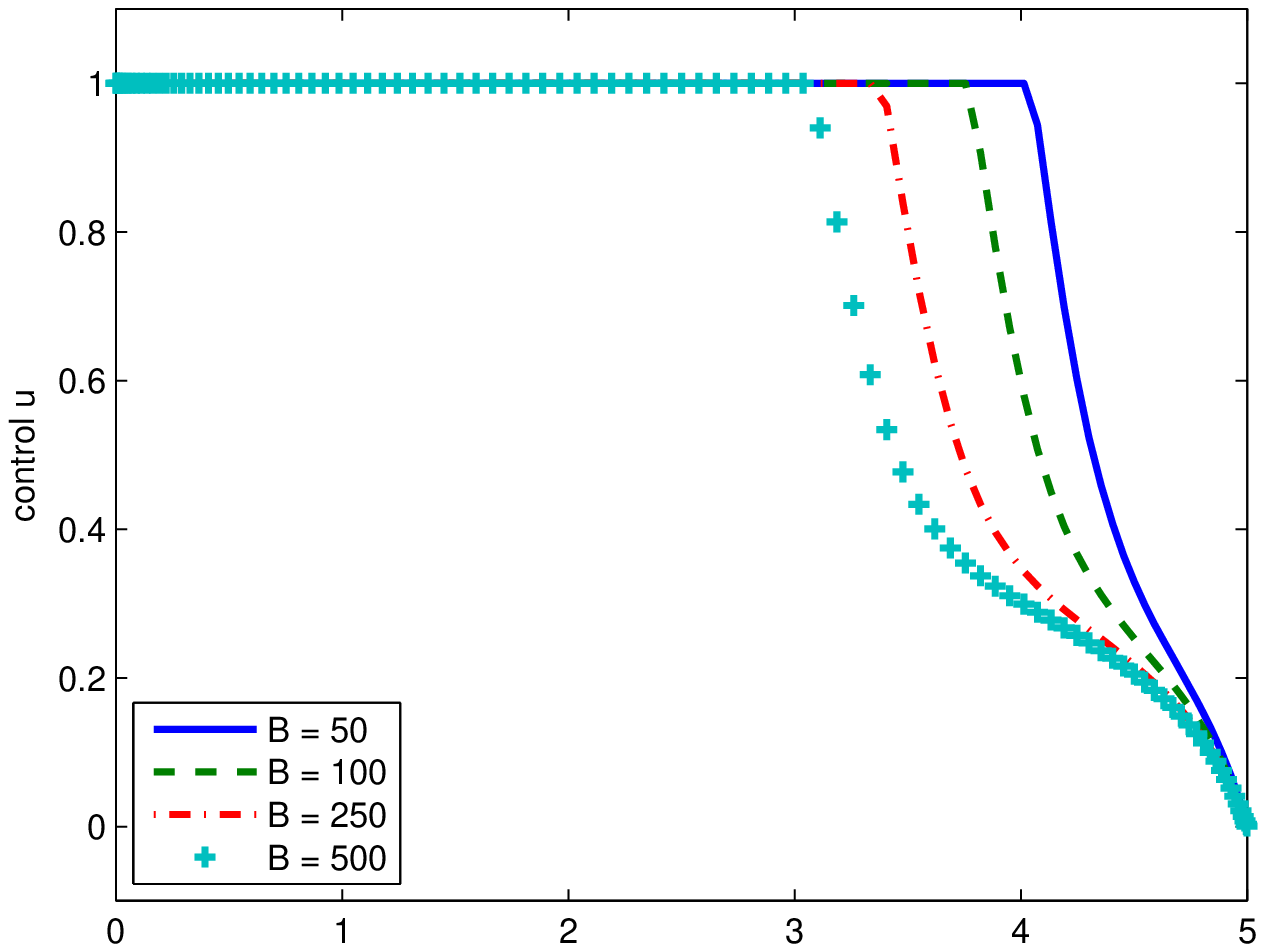}}
\caption{Fraction of infectious individuals and optimal control for
$B \in \{50, 100, 250, 500\}$ (with $A = 1$, $N = 10000$ and $k_1 = 1$).}
\label{fig:I1:u:varB}
\end{figure}


\begin{figure}[!htb]
\centering
\subfloat[\footnotesize{$I_1/N$}]{\label{I1:A100:vark}
\includegraphics[width=0.51\textwidth]{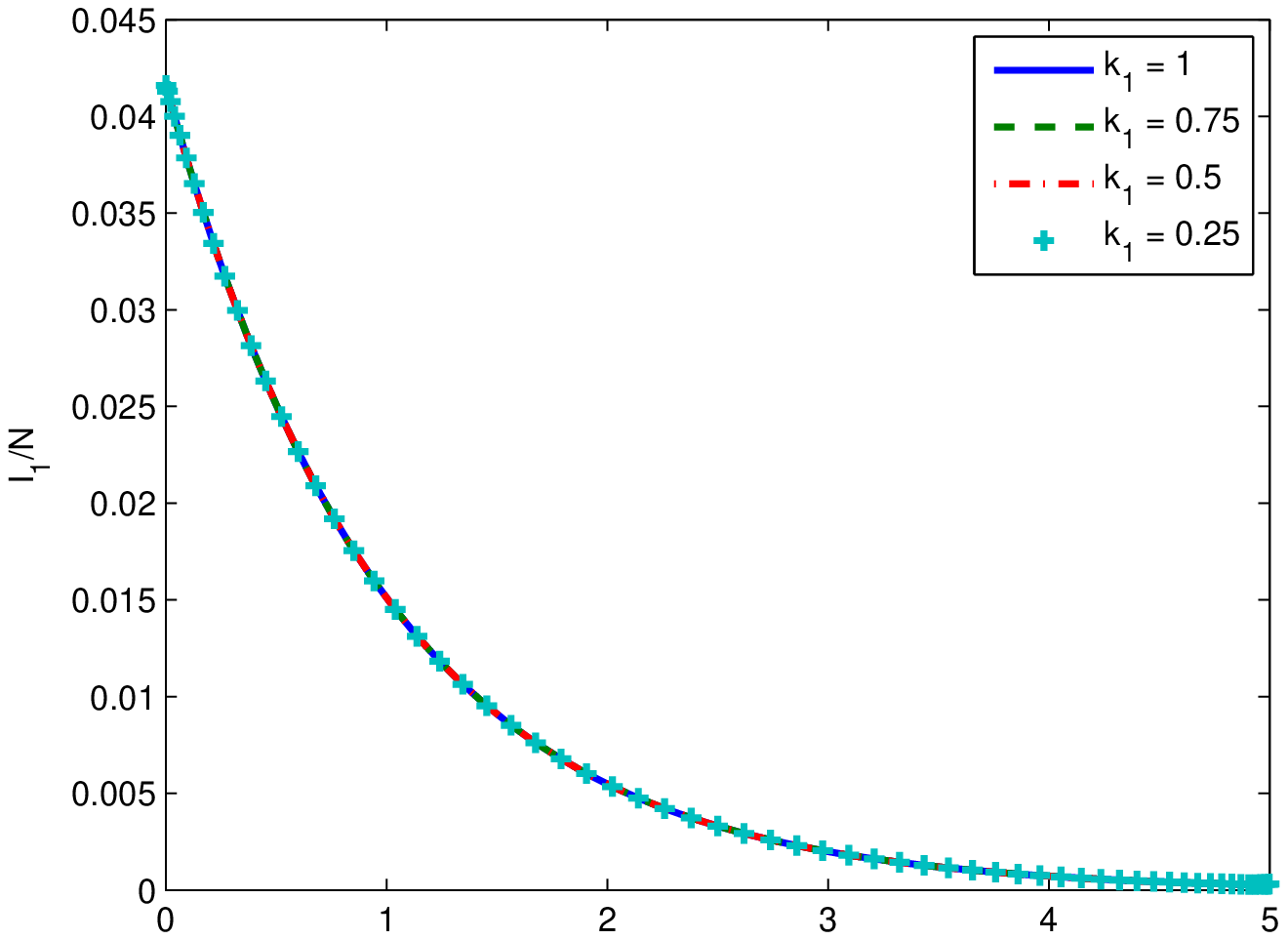}}
\subfloat[\footnotesize{$I_1/N$}]{\label{I1:A100:varN}
\includegraphics[width=0.51\textwidth]{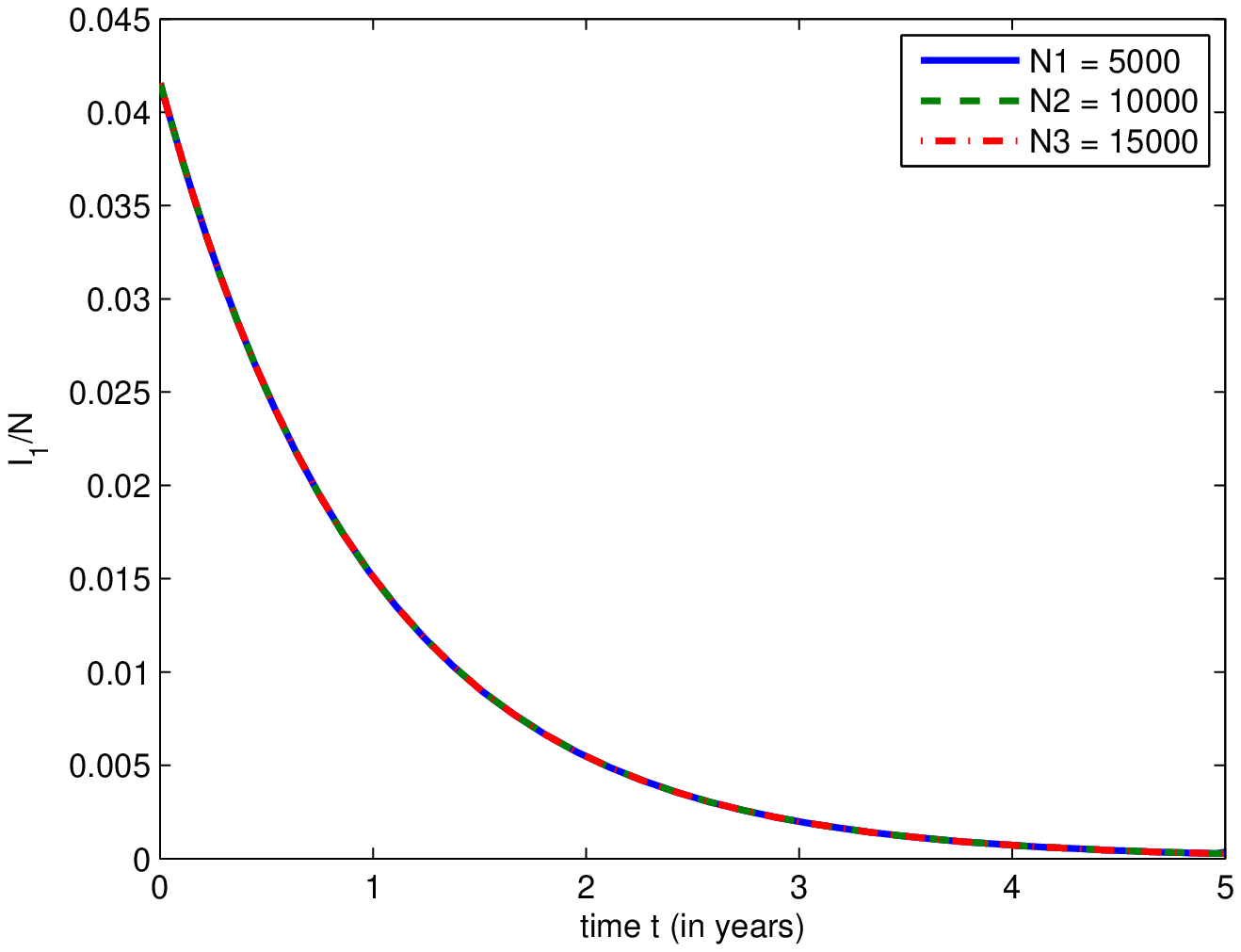}}
\caption{Fraction of infectious individuals for
$A=B=100$ (with $k_1 \in \{0.25, 0.5, 0.75, 1\}$ and $N \in \{5000, 10000, 15000\}$).}
\label{fig:I:A100:varkN}
\end{figure}


\begin{figure}[!htb]
\centering
\subfloat[\footnotesize{$u$}]{
\includegraphics[width=0.51\textwidth]{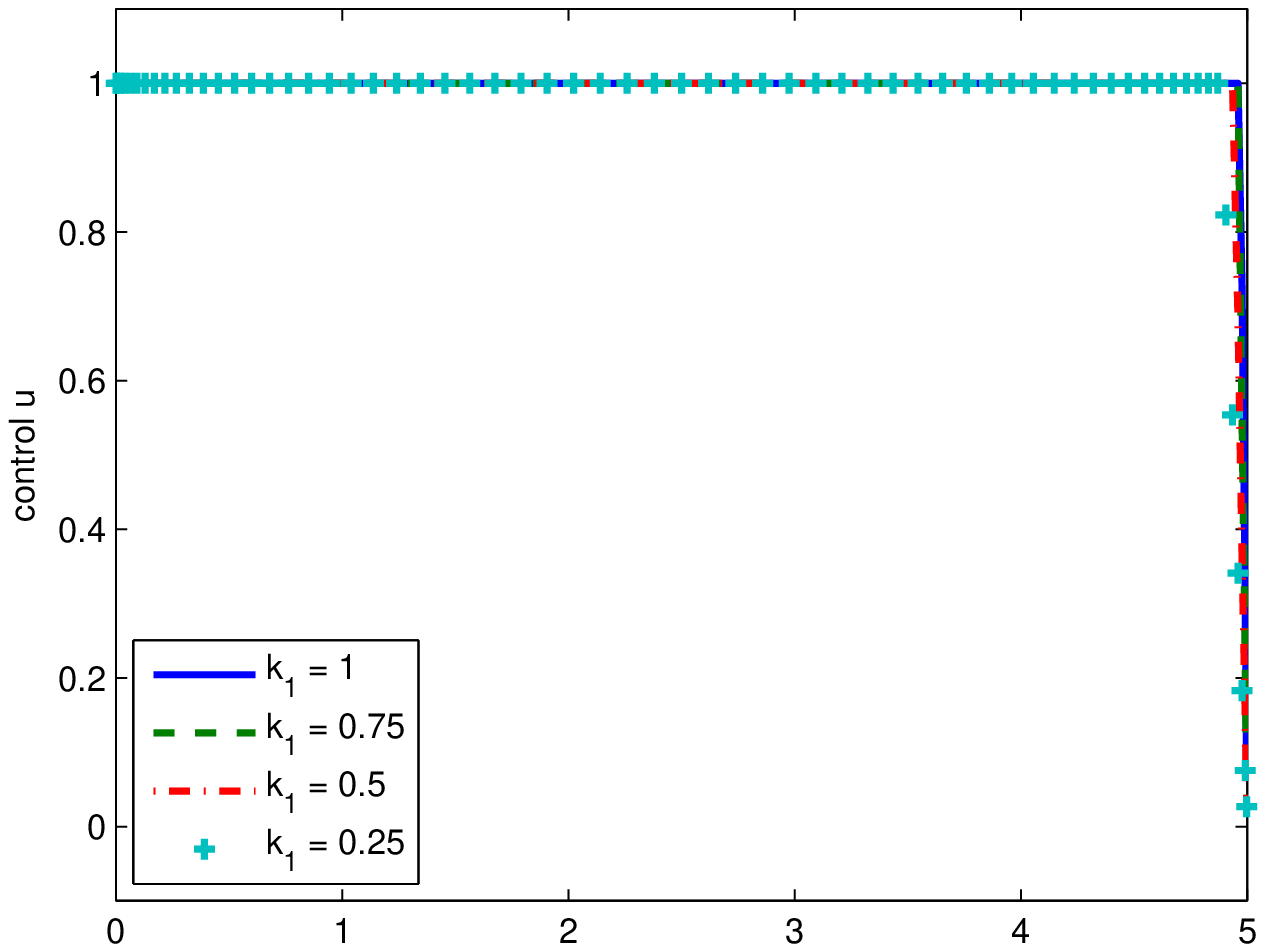}}
\subfloat[\footnotesize{$u$}]{
\includegraphics[width=0.51\textwidth]{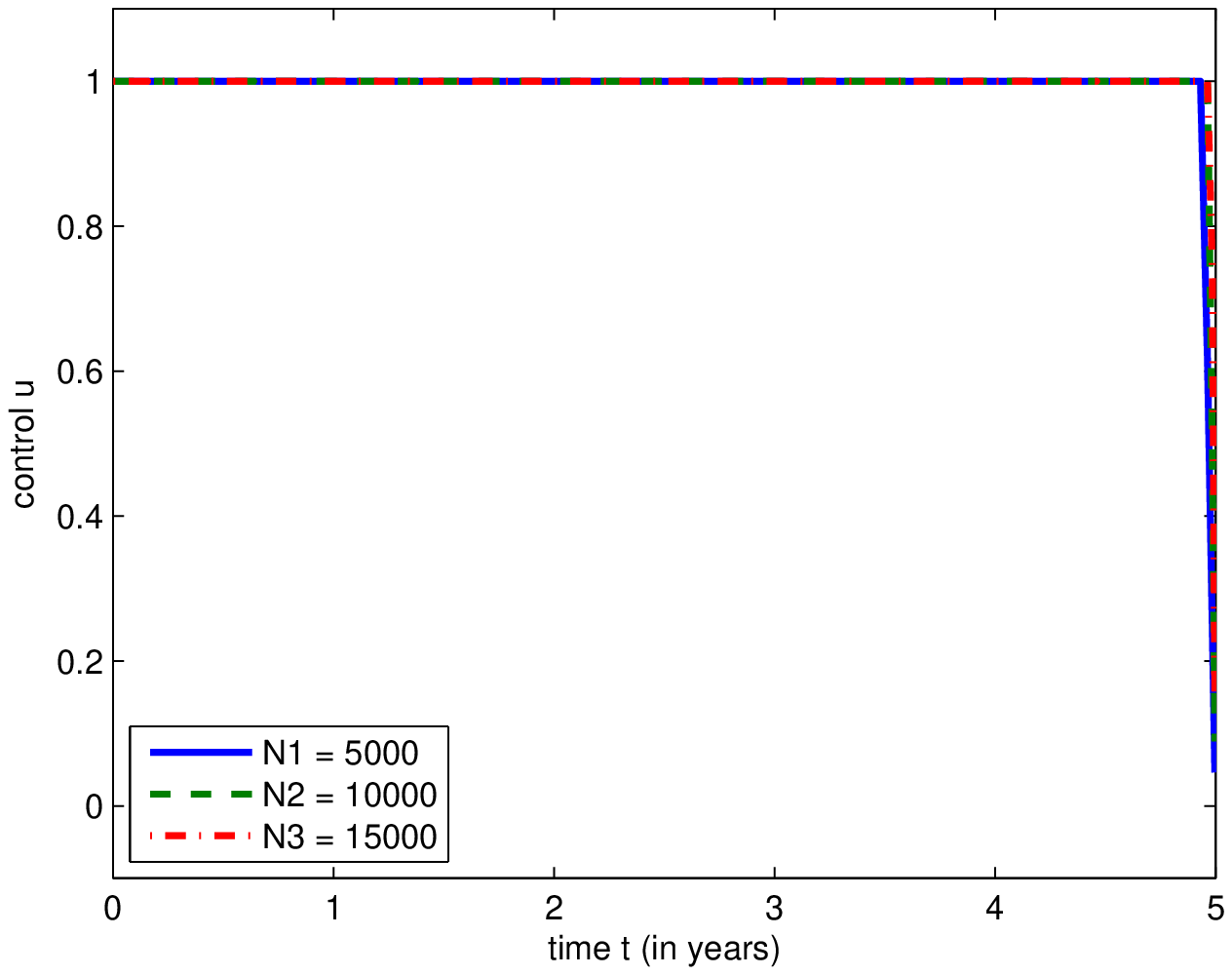}}
\caption{Optimal control for
$A=B=100$ (with $k_1 \in \{0.25, 0.5, 0.75, 1\}$ and $N \in \{5000, 10000, 15000\}$).}
\label{fig:u:A100:varkN}
\end{figure}


\section{Conclusion}
\label{sec:conc}

A state of the art of uncontrolled and controlled
mathematical models for tuberculosis (TB)
has been presented. In particular, the paper reviews
the works on optimal control of various models for the disease transmission
dynamics of TB. Several results related to the dynamics and optimal control
of TB have been reviewed and summarized. Two control strategies,
``case finding'' and ``case holding'', are used to demonstrate the optimal control analysis.

The topics covered do not provide an exhaustive survey
but rather an illustrative overview. For instance, a TB vaccine
called BCG (Bacillus of Calmette and Gu\'{e}rin) has been used especially
for children for several decades, and in some papers a dynamical system with vaccination
has been formulated and analyzed (see, e.g., \cite{Ref2:01}),
but the subject has not been covered here. The example provided
(see Section~\ref{sec:example}) is also very simple: only a single-strain
TB dynamics with SEIRS model is presented. The reader interested
in a model to study the optimal control of a two-strain
(drug-sensitive and drug-resistant) TB dynamics is referred to
\cite{SLenhart_2002}.

Current research includes the development of co-infection mathematical models for TB
and human immunodeficiency virus (HIV) transmission dynamics \cite{MyID:300}.
The novelty of \cite{MyID:300}, with respect to available results in the literature,
is considering both TB and acquired immune deficiency syndrome (AIDS) treatment for individuals
with both infectious diseases. Results show that TB treatment for individuals with only TB infection
reduces the number of individuals that become co-infected with TB and HIV/AIDS,
and reduces the diseases (TB and AIDS) induced deaths. They also show that the treatment
of individuals with only AIDS also reduces the number of co-infected individuals. Further,
TB-treatment for co-infected individuals in the active and latent stage of TB disease,
implies a decrease of the number of individuals that passes from HIV-positive to AIDS.
Application of optimal control to such combined TB-HIV/AIDS co-infection models
poses a number of numerical challenges and is under investigation.
This will be addressed in a forthcoming paper.


\section*{Acknowledgments}

This work was partially presented at the Thematic session
\emph{Control of diseases and epidemics},
MECC 2013 --- International Conference Planet Earth,
Mathematics of Energy and Climate Change, 25-27 March 2013,
Calouste Gulbenkian Foundation (FCL), Lisbon, Portugal.
It was supported by Portuguese funds through the
\emph{Center for Research and Development in Mathematics and Applications} (CIDMA),
and \emph{The Portuguese Foundation for Science and Technology} (FCT),
within project PEst-OE/MAT/UI4106/2014. Silva was also supported by FCT through
the post-doc fellowship SFRH/BPD/72061/2010, Torres by project PTDC/EEI-AUT/1450/2012.
The authors are very grateful to two anonymous referees,
for valuable remarks and comments, which
significantly contributed to the quality of the paper.



\end{document}